\documentclass{article}
\usepackage{amsmath,amssymb,amsthm, mathtools, graphicx,mlspconf}
\usepackage[mathscr]{euscript}
\usepackage[linesnumbered,algoruled,boxed,lined]{algorithm2e}
\SetKwInput{KwInit}{Initialization}
\usepackage{url}
\usepackage{listings}
\usepackage{booktabs} % for professional tables
\usepackage{multirow}
\usepackage{adjustbox}

\newlength{\bibitemsep}
\newlength{\bibparskip}
\let\oldthebibliography\thebibliography
\renewcommand\thebibliography[1]{%
  \oldthebibliography{#1}%
  \setlength{\parskip}{\bibitemsep}%
  \setlength{\itemsep}{\bibparskip}%
}
\usepackage[usenames, dvipsnames]{xcolor}
\definecolor{dkgreen}{rgb}{0,0.6,0}
\definecolor{gray}{rgb}{0.5,0.5,0.5}
\definecolor{mauve}{rgb}{0.58,0,0.82}

\newtheorem{theorem}{Theorem}[section]

\newtheorem{proposition}[theorem]{Proposition}

\theoremstyle{definition}

\theoremstyle{remark}

\DeclareMathOperator*{\argmin}{arg\,min}
\DeclareMathOperator*{\argmax}{arg\,max}

\DeclareMathOperator*{\jaco}{J}
\DeclareMathOperator*{\e}{e}

\newcommand{\x}{w}
\newcommand{\w}{w}
		
\newcommand{\range}[1]{\{1, \dots, #1\}}

\DeclareMathOperator*{\R}{\mathbb{R}}
\DeclareMathOperator*{\RR}{\mathbb{R}}

\DeclareMathOperator*{\Rd}{\mathbb{R}^d}

\DeclareMathOperator*{\Rp}{\mathbb{R}^p}
\DeclareMathOperator*{\Rq}{\mathbb{R}^q}

\copyrightnotice{978-1-7281-6662-9/20/\$31.00 {\copyright}2020 IEEE}

\toappear{2020 IEEE International Workshop on Machine Learning for Signal Processing, Sept.\ 21--24, 2020, Espoo, Finland}

\title{First-order Optimization for Superquantile-based Supervised Learning}
\name{%
    Yassine Laguel$^{\star}$%
    \qquad J\'{e}r\^{o}me Malick$^{*}$%
    \qquad Zaid Harchaoui$^{\dagger}$%
 \thanks{The authors gratefully acknowledge support from NSF CCF 1740551, DMS 1839371, and faculty research awards.}}
\address{%
    $^{\star}$UGA, Lab. J. Kuntzmann, Grenoble, France
 \\%
    $^{*}$CNRS, Lab. J. Kunztmann, Grenoble, France \\%
    $^{\dagger}$ University of Washington, Seattle, USA
}

\begin{document}
%\ninept

%======================================================
% my macro and other personal stuff
%======================================================

\newcommand{\mathbbm}[1]{\text{\usefont{U}{bbm}{m}{n}#1}}
\newcommand{\vertiii}[1]{{\left\vert\kern-0.25ex\left\vert\kern-0.25ex\left\vert #1 
    \right\vert\kern-0.25ex\right\vert\kern-0.25ex\right\vert}}
\newcommand{\trans}[1]{#1^\top}
\newcommand{\jac}{\jaco\!L}

\newcommand{\integerbound}[1]{\{#1 \}}

\maketitle

\begin{abstract}
Classical supervised learning via empirical risk (or negative log-likelihood) minimization hinges upon the assumption that the testing distribution coincides with the training distribution. This assumption can be challenged in modern applications of machine learning in which learning machines may operate at prediction time with testing data whose distribution departs from the one of the training data. 
% We revisit the superquantile approach proposed by\;\cite{rockafellar2008risk} and present a first-order optimization algorithm based on smoothing by infimal convolution to minimize a superquantile-based objective for safer supervised learning. 
We revisit the superquantile regression method by proposing a first-order optimization algorithm to minimize a superquantile-based learning objective. The proposed algorithm is based on smoothing the superquantile function by infimal convolution.
Promising numerical results 
%on synthetic data 
illustrate the interest of the approach towards safer supervised learning.
\end{abstract}
\begin{keywords}
supervised learning; risk measure; distributional robustness; nonsmooth optimization
%infimal convolution smoothing
\end{keywords}
\section{Introduction}
\label{sec:intro}
Classical supervised learning assumes that, at training time, we have access to examples  $(x_1,y_1), \dots, (x_n, y_n)$ drawn i.i.d. from a distribution $\mathbb{P}$, and that at testing time, we may face a new example, also drawn from $\mathbb{P}$. The learned predictor or function can be used by humans or machines to make decisions, or used in as an intermediate component in a greater data processing and computing system.

This common framework is currently challenged by important domain applications~\cite{recht2019imagenet}, in which several of the standard assumptions turn out to be unrealistic or simply incorrect. We may not face the same distribution at test time as we did at training time (train-test distribution shift). Recent failures of learning systems when operating in unknown environments~\cite{metz2018microsoft,knight2018selfdriving} underscore the importance of reconsidering the learning objective used to train learning machines in order to ensure robust behavior in the face of unexpected distributions at prediction time. 
% This can benefit single-machine learning but also multi-machine learning as in federated learning~\cite{laguel2020device}.

The generalized regression framework presented in \cite{rockafellar2008risk} provides an attractive ground to design learning machines displaying increased robustness in the face of unexpected testing distributions. The framework hinges upon the notion of superquantile, a statistical summary of a distribution tail~\cite{lee2018minimax,duchi2019variance,kuhn2019wasserstein}. This notion of robustness is aligned with the one in distributionally robust optimization~\cite{ben2009robust} and empirical likelihood estimation~\cite{owen2001empirical}. It is, however, different, from notions of robustness commonly considered in robust statistics~\cite[Sec. 12.6]{ben2009robust}.
% One would like to have convenient gradient-based optimization algorithms for superquantile-based learning, similar to those developed in recent years for classical supervised learning ~\cite{catalyst}. 

The superquantile is a risk measure, a family of statistical summaries of distribution tails, well studied in economics and finance~\cite{rockafellar2000optimization,ben2007old}.
%\cite{artzner1999coherent,rockafellar2000optimization,ben2007old}. 
The quantity is, however, a nonsmooth function. 
We present here a simple approach, based on infimal convolution smoothing, which allows one to easily adapt state-of-the-art gradient-based optimization algorithms for classical supervised learning to the superquantile-based learning framework. Moreover, we provide a companion software package in Python available here~\url{https://github.com/yassine-laguel/spqr}~. %along these lines. 
\subsection{Superquantile}\label{sec:setting}
% ===========================================================

Risk measures play a crucial role in optimization under uncertainty, involving problems with an aversion to worst-cases scenarios. Among popular convex risk measures, superquantile (also called Conditional Value at Risk) has received a special attention because of its nice convexity properties; see e.g.\;the
%we refer to the seminal work\;\cite{rockafellar2000optimization} and the classical 
textbook~\cite[Chap.\;6]{shapiro2014lectures}. 
% Superquantiles have been used in various applications.
% , e.g.\;in finance \cite{sarykalin2008value}, energy planning \cite{guigues2012Risk}, or machine learning \cite{tamar2015policy}. 

We use here the notation and terminology of Rockafellar and Royset\;%follow the ones of \cite{ruszczynski2006optimization} and
\cite{rockafellar2013superquantiles}. 
%We consider a probability space $(\Omega, \mathcal{A}, \mathbb{P})$ and the set $\mathcal{L}^2(\Omega, \mathcal{A}, \mathbb{P})$ of real random variables on $\Omega$ that admit a second order moment. 
The $p$-quantile $Q_p(U)$ of a random variable $U$ is defined as the general inverse of the cumulative distribution of $U$. More precisely, for a random variable $U$ (admitting a second order moment),
%\in \mathcal{L}^2$, 
the cumulative distribution function $F_U\colon \mathbb{R} \rightarrow [0,1]$ is defined as $F_U(x)=\mathbb{P}(U \leq x)$. For any $p \in [0,1]$, the $p$-quantile $Q_p(U)$ and the $p$-superquantile $\bar{Q}_p(U)$, are respectively defined by
\begin{equation}\label{eq:def_cvar}
\begin{split}
Q_p(U)  &= \min \{x \in \mathbb{R}, F_U(x) \geq p  \}  \\
\bar{Q}_p(U) &= \frac{1}{1-p} \int_{p^\prime=p}^1 Q_{p^\prime}(U) \mathrm{d}p^\prime. \\
\end{split}
\end{equation}
% This function has been extensively studied from a convex analysis perspective:
% we refer for instance to \cite{rockafellar2000optimization} for a variational formulation of the superquantile, to \cite{ben2007old} for its generalization to a larger class of risk measures, to \cite{follmer2002convex} for a dual formulation (also later generalized in \cite{ruszczynski2006optimization} or \cite{rockafellar2002conditional}) and \cite{rockafellar2014random} for additional convex properties.
The superquantile is, therefore, 
a measure of the upper tail. The parameter $p$ allows one to control the sensitivity to risk. The 
superquantile enjoys a dual representation\;\cite{follmer2002convex}
\begin{equation}\label{eq:def_max_cvar}
\bar{Q}_p(U) = \max_{{\substack{0\leq q(\cdot) \leq \frac{1}{1-p}\\ \int_\Omega q\, d\,\mathbb{P}(\nu) = 1}}} \int_{\nu \in \Omega} U(\nu) q(\nu) \mathrm{d} \mathbb{P}(\nu) 
\end{equation}
Interestingly, the dual formulation uncovers another interpretation of the superquantile learning objective relating it to the re-weighting of the terms in the empirical risk. In practice, the ambiguity on the data distribution may be formalized before training, for instance by incorporating side information (geographical and/or temporal for instance) that drives the heterogeneity of the data. Superquantile learning is expected to produce models that perform better in case of distributional shifts between the training time and the testing time, compared to models trained using standard empirical risk minimization. 
%We will illustrate this in our numerical experiments.

% \begin{itemize}
% \item dual formulation \cite{follmer2002convex}
% \begin{equation}\label{eq:def_max_cvar}
% \bar{Q}_p(U) = \max_{{\substack{q \in \mathcal{L}^2\\ 0\leq q(\cdot) \leq \frac{1}{1-p}\\ \int_\Omega q d\,\mathbb{P}(\omega) = 1}}} \int_{\omega \in \Omega} U(\omega) q(\omega) \mathrm{d} \mathbb{P}(\omega) 
% \end{equation}
% \item one-dimensional formulation \cite{rockafellar2000optimization}
% \begin{equation}\label{eq:def_min_cvar}
% \bar{Q}_p(U) = \min_{u \in \mathbb{R}} ~~u + \frac{1}{1-p} \mathbb{E}[\max(U-u,0)] 
% \end{equation}
% %\item conditional expectation formulation see section 2.
% \end{itemize}
% We also note that $Q_p(U)$ can be obtained as the left end-point of solution set to \eqref{eq:def_min_cvar}. 

%=================================================================
% \subsection{Learning via superquantile}
\subsection{Superquantile-based learning}
%=================================================================

We are interested in a supervised machine learning setting with training data $\mathcal D = (x_i,y_i)_{1 \leq i \leq n} \in (\Rp \times \Rq)^n$, 
a prediction function $\varphi: \Rd \times \Rp \rightarrow \Rq$ (such as a linear model or a neural network) and a loss function $\ell: \Rq \times \Rq \rightarrow \R$ (such as the logistic loss or the least-squares loss).  
The classical empirical risk minimization writes
\begin{equation}\label{eq:ERM}
\min_{w \in \Rd} \mathbb{E}_{(x_i,y_i)\sim \mathcal{D}}\left(\ell(y_i, \varphi(w,x_i )\right).
\end{equation}
% The robust optimization viewpoint is in contrast with robust statistics: we assume that $\mathcal D$ is preprocessed, without outliers and with relevant extreme data that we want primarily to take into account in the learning process. 

A natural approach consists then in replacing the expectation in \eqref{eq:ERM} by the superquantile \eqref{eq:def_cvar} in the case of discrete distributions standing for the training data
\begin{equation}\label{eq:ESM}
\min_{w \in \Rd} {[{\bar Q}_{p}]}_{(x_i,y_i)\sim \mathcal{D}}\big(\ell(y_i, \varphi(w,x_i )\big)
\end{equation}
Introducing $L^i(w) =\ell(y_i, \varphi(w,x_i))$ and 
$L(w)= (L^i(w))_i$, we simply write the 
superquantile optimization problem as
% has with a %convex 
% objective function that is the composition of the superquantile risk measure and a loss function:
\begin{equation}\label{eq:general_problem}
\min_{\x \in \mathbb{R}^d} ~~f(\x) = \bar{Q}_p(L(\x)).
\end{equation}

Note that the objective function can specified by expressing the superquantile in its dual formulation \eqref{eq:def_max_cvar} for the discrete distribution
\begin{align*}\label{eq:def_max_cvar_discrete}
%\bar{Q}_p(L(\x))
f(\w)&= \sup_{q \in K_p} \sum_{i=1}^n q_i L^i(\x) \quad \text{with}\\
 K_p &= \left\{q \in \mathbb{R}^n, \sum_{i=1}^n q_i = 1,  q_i\in \left[0, \frac{1}{n(1-p)}\right]\; \forall i  \right\}.
\end{align*}
This representation is central to the implementation as we shall see in Sec. 3. 
% in practice in the toolbox minimizing~\eqref{eq:general_problem}.
%from discrete data.
%\textcolor{red}{[Remark on Top-k function here ?]}
% In practice, superquantile-based problems are often solved by specific methods. For instance, the superquantile formulation of quantile regression leads to linear-programming approaches.
Existing works on minimizing  
superquantiles considered linear programming 
or convex programming 
including interior point algorithms; see\;\cite{rockafellar2014superquantile}.
Our approach considers first-order algorithms instead; although natural, this work seems to be the first one to do so.

%Un-necessary text; repetition from intro.
%Add more math. content instead
% %=================================================================
% \subsection{Contributions and outline}
% %=================================================================

% XXX
% Introduction of superquantile-based ?

% We release \texttt{SPQR}, a publicly available toolbox to solve general superquantile optimization problems.
% We explain the theory on which it is grounded, the key numerical ingredients, and how to use it. 

% % We propose to use within \texttt{SPQR} standard first-order algorithms for general superquantile-based optimization problems. 
    
% We primarily focus on operational aspects; the study of related statistical properties are postponed to further research.

% This paper is structured as follows.  In section \ref{sec:sub-grad-expr}, we study the differentiability of the objective of \eqref{eq:general_problem}, provide practical expressions of its (sub)gradients, and present smoothing approach based on \cite{nesterov2005smooth} together with fast computational procedures. We provide in Section\;\ref{sec:spqr} a short presentation of the toolbox \texttt{SPQR} using these fast oracles with standard first-order methods. Finally, we report in Section\;\ref{sec:numexp} numerical illustrations on a basic regression task with both a didactic synthetic framework and a real dataset.

\section{Smoothing the Superquantile}\label{sec:sub-grad-expr}
In this section, we study the differentiability properties of the superquantile objective~\eqref{eq:general_problem}. We first derive the expression of the subdifferential, when the $L^i$-s are convex\footnote{Convexity of the $L^i$-s is guaranteed when e.g.\;the model $\varphi$ is linear and the loss $\ell$ is convex with respect to its second variable, as for the $l_2$-squared loss and the the cross-entropy loss.}. Then, when $L^i$ are smooth (and possibly nonconvex), we show how to smooth the superquantile
%using its dual representation 
by infimal convolution in order to apply gradient-based optimization algorithms~\cite{nesterov2005smooth}.

% \yassine{Note that the convexity assumption on the functions $L^i$ is satisfied whenever $L^i$ is of the form $w\mapsto \psi(w, x_i, y_i)$ with $\psi$ being convex with respect to $w$. It thus encompass the majority of popular convex learning losses including:
% \begin{itemize}
%     \item : 
%     \[
%         \psi(w, x_i, y_i) = \|y_i - w^\top x_i\|^2
%     \]
%     \item : 
%     \[
%     \begin{split}
%         \psi(w, x_i, y_i) &= - y_i \log(\hat{y_i}) - (1-y_i) \log(1-\hat{y_i}) \\
%          \text{with } \hat{y_i} &= 1/(1+e^{-\w^\top x_i}) \\
%     \end{split}
%     \]
% \end{itemize}}

%====================================================================
\subsection{Subdifferential expression}
%====================================================================

The superquantile risk measure\;\eqref{eq:general_problem}
%, as a composition of the nonsmooth superquantile function with a convex loss,
is usually nonsmooth and computing its subdifferential (or even a single subgradient) is not straightforward. 
Using the dual reformulation\;\eqref{eq:def_max_cvar}, we get the expression of the entire subdifferential for the convex case.
Note that gradients of superquantile-based functions for general distributions are obtained, with advanced tools, in~\cite{ruszczynski2006optimization}. Interestingly, the nonsmoothness of these functions arises only with discrete distributions.
% by applying the general result of \cite{ruszczynski2006optimization} for convex risk functions in our case. 

\begin{proposition}\label{thm:sub}
Assume %$L(\x)$ follows a discrete distribution and 
the model $\varphi$ and the loss $\ell$ are such that the $L^i$ are convex. For $\x \in \mathbb{R}^d$, let $I_p(\x)$ be the set of indices $i \in \{1,\dots, n\}$ such that $L^i(\x) = Q_p(L(\x))$. Then the subdifferential reads as a Minkowski sum
\begin{equation*}
\begin{split}
    \partial f(\x) & = \frac{1}{1-p} \sum_{\substack{i \in \{1,\dots n\} \\ L^i(\x) > Q_p(L(\x))}} \frac{\partial L^i(\x)}{n} \\ 
    & + \bigg\{  \frac{1}{1-p} \sum_{i \in I_p(\x)} \alpha_i \frac{\partial L^i(\x)}{n} , \alpha_i \in [0,1]\, \forall i \in I_p(\x)\\
    & \frac{1}{n} \sum_{i\in I_p(\x)} \alpha_i = \frac{1}{n}\sum_{i=1}^n \mathbbm{1}_{L^i(\x) \leq Q_p(L(\x))}  - p  \bigg\}
\end{split}
\end{equation*}
In particular, when $L$ is differentiable at $\x$, $f$ is differentiable at\;$\x$ if and only if the set  $I_p(\x)$ is reduced to a singleton. 
\end{proposition}

\begin{proof}
The proof consists in applying various convex calculus rules, taken from the textbook~\cite[Chap\;D]{hiriart2013convex}. First 
%For any $i$, 
we apply
Theorems\;4.1.1 and\;4.4.2 % of\;\cite{hiriart2013convex}
to $h_i(\x, \eta)=\max(L^i(\x) - \eta)$
% the function $h_i:(\x, \eta) \mapsto \max(L(\x)_i - \eta)$ is subdifferentiable on $\R^d$ with: 
\begin{equation*}
    \partial h_i(x, \eta) = \{(\partial L^i(\x), -1) (\mathbbm{1}_{L^i(\x) > \eta} + \alpha \mathbbm{1}_{L^i(\x) = \eta}),\; \alpha \in [0, 1]\}
\end{equation*}
We apply
Theorem\;4.1.1 %of\;\cite[Chap\;D]{hiriart2013convex} 
with $h(\x, \eta) =\eta + \frac{1}{n(1-p)} \sum_{i=1}^n h_i(\x, \eta)$
% the subdifferentiability of $h$ on $\mathbb{R}^d$ with:
% \begin{equation*}
% \begin{split}
%     \partial h(\x, \eta) =  &\left\{ \left(\frac{1}{1-p} \sum_{i=1}^n \frac{\partial L^i(\x)}{n} (\mathbbm{1}_{L(\x)_i > \eta} + \alpha_i \mathbbm{1}_{L^i(\x) = \eta}),  \right. \right.\\
%     &  \left. 1 - \frac{1}{1-p} \sum_{i=1}^n \frac{1}{n} (\mathbbm{1}_{L^i(\x) > \eta} + \alpha_i \mathbbm{1}_{L^i(\x) = \eta}) \right), \\
%     & \left. \alpha_i \in [0,1], \quad \forall i \in \range{n}\right \Bigg\}. \\
% \end{split}
% \end{equation*}
\begin{equation*}
\begin{split}
    \partial h(\x, \eta) =  &\left\{ \left(\frac{1}{1-p} \sum_{i=1}^n \frac{\partial L^i(\x)}{n} (\mathbbm{1}_{L(\x)_i > \eta} + \alpha_i \mathbbm{1}_{L^i(\x) = \eta}),  \right. \right.\\
    &  \left. 1 - \frac{1}{1-p} \sum_{i=1}^n \frac{1}{n} (\mathbbm{1}_{L^i(\x) > \eta} + \alpha_i \mathbbm{1}_{L^i(\x) = \eta}) \right), \\
    & \left. \phantom{\sum_{i=1}^n} \alpha_i \in [0,1], \quad \forall i \in \range{n} \right\}. \\
\end{split}
\end{equation*}
By \cite{rockafellar2000optimization}, $f$ satisfies  $f(\x)\!=\!\min_{\eta \in \mathbb{R}} h(\x, \eta)$, with $Q_p(L(w))\!=\!\argmin_{\eta \in \mathbb{R}} h(\x, \eta)$. We can thus apply Corollary\;4.5.3\;to\;get
%the subdifferential of $f$:
\begin{equation*}
\begin{split}
    \partial f(\x) =  &\left\{\frac{1}{1-p} \sum_{i=1}^n \frac{\partial L^i(\x)}{n}   \delta^i(w,\alpha)~~~\text{with $\alpha$ s.t.} \right.\\
    & \left. 0 = 1 - \frac{1}{1-p} \sum_{i=1}^n \frac{\delta^i(w,\alpha)}{n} \text{ and }   \alpha_i \in [0,1], \forall i %\in \range{n}
    \right\} \\
\end{split}
\end{equation*}
with $\delta^i(w,\alpha) = (\mathbbm{1}_{L^i(\x) >  Q_p(L(\x))} + \alpha_i \mathbbm{1}_{L^i(\x) =  Q_p(L(\x))})$.
Observe finally that for any sequence $(\alpha_i)_{1\leq i \leq n}$
\begin{equation*}
\begin{split}
     0 = 1 - & \frac{1}{1-p} \sum_{i=1}^n \frac{\delta^i(w,\alpha)}{n}\\
     %(\mathbbm{1}_{L^i(\x) > Q_p(L(\x))} + \alpha_i \mathbbm{1}_{L^i(\x) = Q_p(L(\x))})\\
    \Leftrightarrow  \frac{1}{n}\sum_{i \in \mathcal{I}_p(z)} \alpha_i &= 1 - p - \sum_{i=1}^n \frac{1}{n} \mathbbm{1}_{L^i(\x) >  Q_p(L(\x))}  \\
    \Leftrightarrow \frac{1}{n} \sum_{i \in \mathcal{I}_p(z)} \alpha_i &= 1 - p - (1 - \mathbb{P}[L(\x) \leq  Q_p(L(\x))])\\
    \Leftrightarrow  \frac{1}{n} \sum_{i \in \mathcal{I}_p(z)} \alpha_i &= \frac{1}{n}\sum_{i=1}^n \mathbbm{1}_{L^i(\x) \leq Q_p(L(\x))} - p  \\
\end{split}
\end{equation*}
which yields the result.
\end{proof}

Thus, the computation of a subgradient can be performed in linear time: the cost essentially stems from the computation of the quantile $Q_p(L(w))$ and the sum of vectors in $\RR^d$ (assuming such sums can be computed in constant time).

%=================================================================
\subsection{Gradient of smoothed approximation}\label{sec:smooth_approximation}
%=================================================================

As shown in Proposition~\ref{thm:sub}, the objective function is not differentiable in general (even when $L$ is differentiable), and we propose to smooth it using infimal convolution as in~\cite{nesterov2005smooth}. More precisely, we
follow the methodology of~\cite{doi:10.1137/100818327} and we propose to smooth only the superquantile $\bar Q_p$ rather than the whole function\;$f$. Given formulation \eqref{eq:def_max_cvar}, we introduce% the function % $f_\mu$ as
%the composition of $L$ by the standard smoothing of $\bar Q_p$, which reads
\begin{equation}\label{eq:def_f_mu}
    f_{\mu}(\x) = \max_{q \in K_p} \sum_{i=1}^n q_i \; L^i(\x)
    - \mu\ d(q) \qquad\text{for $\mu>0$} 
\end{equation}
where $d:\mathbb{R}^n \rightarrow \mathbb{R}$ is a fixed non-negative strongly convex function that satisfies $\min_{q \in K} d(q) = 0$.
In this paper, we consider the euclidean distance to the uniform probability measure and the entropic penalty function
\begin{equation*}%\label{eq:div}
d(q) = \frac{1}{2} \left\|q - \frac{1}{n}\e\right\|^2 
\quad\text{and}\quad
d(q) = \log(n) + \sum_{i=1}^n q_i \log(q_i) 
\end{equation*}
where $e=(1, \dots,1)^\top$ is the usual vectors of all ones.
As a direct application of~\cite[Th.\;1]{nesterov2005smooth}, we have the following proposition establishing that $f_\mu$ is a smooth approximation of $f$.

\begin{proposition}[Gradient of smoothed approximation]\label{thm:smooth_approx}
Assume the model $\varphi$ and the loss $\ell$ are such that the $L^i$ are smooth for any $i$. 
In the above setting, the convex function $f_\mu$ provides a global approximation of $f$, i.e.
$f_\mu(\x) \leq f(\x) \leq  f_\mu(\x) + \frac{\mu}{2}$ for any $\x \in \mathbb{R}^d$.
If $L$ is differentiable, then $f_\mu$ is differentiable as well, with 
\begin{equation}\label{eq:grad_phi_mu}
\nabla f_\mu(\x) = \jac(\x)^T q_\mu(\x),
\end{equation}
where $\jac(\x)$ is the Jacobian of $L$ at $\x$ and $q_\mu(\x)$ is the optimal solution of~\eqref{eq:def_f_mu}, unique by
strong convexity of $d$.
\end{proposition}

%\yassine{We note here that no convex assumption is required on the functions $L^i$.}
To be made practical, the previous result needs to be equipped with a fast and efficient procedure to solve~\eqref{eq:def_f_mu}. As stated in the next proposition, Algorithm\;\ref{algo_projection} addresses this issue. The procedures follows closely the ones in~\cite{condat2016fast}, where convex duality and one-dimensional search ideas are fruitfully combined.
%\cite{PIE2017} and

\begin{algorithm}[t!]
\KwInit{$u= L(\x) + \frac{\mu}{n} \e$, ~~$\ell = \frac{1}{n(1-p)}$,~~$q_\mu = 0\in \mathbb{R}^n$ $\mathcal{P} = \{u_i, i \in \{1,\dots,n\}\} \cup \{u_i -  \mu \ell, i \in \{1,\dots,n\}\}$
%    \item $\theta'$ : function as defined in \eqref{eq:def_theta_prime}
}

\vspace*{-1.5ex}
Find $a := \max \left\{s \in \mathcal{P},\theta'(s) \leq 0 \right\}$\\ \,~~~~~~~~$b := \min \left\{s \in \mathcal{P}, \theta'(s) > 0 \right\}$\;
    \uIf{$\theta'(a) = 0$}{
        $\lambda := a$\;        
        }
    \uElse{
        $\lambda := a - \frac{\theta'(a)(b-a)}{\theta'(b) - \theta'(a)}$
        }
\For{$1 \leq k \leq n$}{
  \uIf{$\lambda < u_k - \mu \ell$}{
    ${[q_\mu]}_k = \ell$\;
  }
  \uElseIf{$u_k - \mu \ell \leq \lambda < u_k$} {
    ${[q_\mu]}_k = \frac{u_k - \lambda}{\mu}$\;
  }
  \Else{
    ${[q_\mu]}_k = 0$
  }
  }
\KwOut{$q_\mu \in \mathbb{R}^n$ : solution of \eqref{eq:def_f_mu}}
  
\caption{Fast subroutine for smoothed oracle}% the gradient of the smoothed counterpart}
\label{algo_projection}
\end{algorithm}

\begin{proposition}\label{thm:On}
    Algorithm \ref{algo_projection} computes the optimal solution of the problem \eqref{eq:def_f_mu} (with the euclidean or the entropic penalty) at a cost of $\mathcal{O}(n)$ operations.
\end{proposition}

\begin{proof}
We detail the proof for $d(q)=\frac{1}{2} \|q \;- 1/n\,\e \|^2$; the second 
%treat briefly the 
case of the entropy follows the same lines. 
We dualize the constraint $\sum_{{i=1}}^n q_i - 1 = 0$ to get the Lagrangian:
\begin{equation*}
\mathscr{L}(q,\lambda) = 
\sum_{{i=1}}^n q_i L^i(\x) - \frac{\mu}{2} \sum_{{i=1}}^n \left(q_i - \frac{1}{n}\right)^2 + \lambda \left(1 - \sum_{{i=1}}^n q_i\right).
\end{equation*}
% \begin{equation*}
% L(q,\lambda) = \sum_{{i=1}}^n q_i u_i - \frac{\mu}{2}(\sum_{{i=1}}^n q_i^2) - \frac{\mu}{2n} + \lambda (1 - \sum_{{i=1}}^n q_i)
% \end{equation*}
% % In order to enlighten the notations, and for the rest of the proof, let us denote by :
% \begin{equation*}
% \begin{split}
% l &:= \frac{1}{n(1-p)} \\
% u &:= L(\x) + \frac{\mu}{n} \e 
% \end{split}
% \end{equation*}
With the notation $\ell$ and $u$ introduced in the algorithm, the dual function writes:
\begin{equation*}
\theta(\lambda) = \max_{\substack{q \in \mathbb{R}^n \\ 0 \leq q_i \leq l}} \mathscr{L}(q,\lambda) =  \lambda - \frac{\mu}{2n} + \sum_{{i=1}}^n \max_{{0 \leq q_i \leq l}} (u_i - \lambda)q_i - \frac{\mu}{2} q_i^2 
\end{equation*}
For $\lambda \in \mathbb{R}$ and $i \in \{1,\dots, n\}$ fixed, let us introduce the function $h_i(q_i) = (u_i - \lambda)q_i - \frac{\mu}{2} q_i^2$. 
%Then, by a simple study of the function $h_i$, one gets
Then, we get
\begin{equation}\label{eq:sol_q_i}
\begin{split}
    \argmax_{{0 \leq q_i \leq l}}h_i(q_i) &=  \left\{
    \begin{array}{lll}
        0 &\mbox{ if } \lambda \geq u_i \\
        \frac{u_i - \lambda}{\mu} &\mbox{ if } u_i \geq \lambda \geq u_i - \mu \ell \\
        \ell  &\mbox{ if } \lambda \leq u_i - \mu \ell  \\
    \end{array}
    \right.        
    % \max_{{0 \leq q_i \leq l}}h_i(q_i) &=  \left\{
    % \begin{array}{lll}
    %     0 &\mbox{ if } \lambda \geq u_i \\
    %     \frac{(u_i - \lambda)^2}{2 \mu} &\mbox{ if } u_i \geq \lambda \geq u_i -  \mu \ell \\
    %     l(u_i - \lambda) - \frac{\mu}{2} l^2  &\mbox{ if } \lambda \leq u_i - \mu \ell  \\
    % \end{array}
    % \right. \\       
\end{split}
\end{equation}
As a result, we get the explicit expression of $\theta(\lambda)$. Observing that it is differentiable, we get
% \begin{equation*}
%     \theta(\lambda) = \lambda - \frac{\mu}{2n} + \sum_{{i=1}}^n \frac{(u_i - \lambda)^2}{2 \mu} \mathbbm{1}_{u_i \geq \lambda \geq u_i - \mu \ell} + (\ell(u_i - \lambda) - \frac{\mu}{2} \ell^2) \mathbbm{1}_{u_i - \mu \ell > \lambda}.
% \end{equation*}
% Moreover $\theta$ is differentiable with, for any $\lambda \in \mathbb{R}$,
\begin{equation*}
    \theta'(\lambda) = 1 - \sum_{{i=1}}^n \left(\frac{u_i - \lambda}{\mu} \mathbbm{1}_{u_i \geq \lambda \geq u_i - \mu \ell} + \ell \mathbbm{1}_{u_i - \mu \ell > \lambda}\right).
\end{equation*}
Observe now that $\lim_{{\lambda \rightarrow + \infty}} \theta'(\lambda) = 1$ and since $n\ell =\frac{1}{1-p} > 1$,  $\lim_{{\lambda \rightarrow - \infty}} \theta'(\lambda) < 0$. Therefore, $\theta'$ is a non-decreasing and continuous (piecewise affine) function that takes negative and positive values: by the intermediate value theorem, there exists a solution $\lambda^\star \in \mathbb{R}$ such that $\theta'(\lambda^\star) = 0$. By duality theory, the associated $q^\star$ (the optimal solution of\;\eqref{eq:sol_q_i} for $\lambda= \lambda^\star$) is the solution of the primal problem\;\eqref{eq:def_f_mu}.
Finally, we compute $\lambda^\star$ zeroing $\theta'$. Since $\theta'$ is piecewise affine, we just need to evaluate  $\theta'$ at points belonging to the set $\mathcal{P}$ and at $a$ and $b$ as defined in Algorithm\;\ref{algo_projection}. One can then find $\lambda^\star$ by testing three simple cases (i)
if $\theta'(a) = 0$, take $\lambda^* = a$, 
if $\theta'(b) = 0$, take $\lambda^* = b$,
else, take $\lambda^* = a - \frac{\theta'(a)(b-a)}{\theta'(b) - \theta'(a)}$.

Regarding computational costs, this algorithm boils down to the search of $a$ and $b$, and the assignment of the coordinates of $q_\mu$. This also sums up to a $\mathcal{O}(n)$ cost.
\end{proof}

% The same arguments apply when considering the entropic penalty. Indeed, we can again dualize the simplex constraint to obtain the separable lagrangian 
% \begin{equation*}
% \begin{split}
%     \mathscr{L}(q,\lambda) = \lambda + \sum_{i=1}^n q_i (x_i - \lambda) - \sum_{i=1}^n q_i \log(q_i) - \log(n). 
% \end{split}
% \end{equation*}
% The derivative of the associated dual function $\theta$ is
% \begin{equation*}
%     \theta'(\lambda) = 1 - \!\!\!\!\!\!\!\!\!\!\!\!\!\!\!\! \sum_{\substack{i \in [1,\dots, n] \\ \lambda > x_i + \log(n(1-p)) - 1}} \!\!\!\!\!\!\!\!\!\!\!\!\!\!\!\! e^{x_i - (1+ \lambda)} - \!\!\!\!\!\!\!\!\!  \sum_{\substack{i \in [1,\dots, n] \\ \lambda \leq x_i + \log(n(1-p)) - 1}} \!\!\!\!\!\!\!\!\!\!\!\!\!\!\!\! \frac{1}{n(1-p)}
% \end{equation*}
% As before, $\theta'$ is a continuous non-decreasing map, and finding a zero of $\theta'$ helps recovering a primal solution which writes 
% \begin{equation*}
%     q_i^{\star} = \min\left\{\frac{1}{n(1-p)}\;, e^{x_i - (1 + \lambda)}\right\} \quad \forall i \in \{1,\dots,n\}
% \end{equation*}
% in this case. Again the cost is in $\mathcal{O}(n)$ (search of $a$ and $b$, and coordinates of $q_\mu$).
% 

Thus Algorithm\;\ref{algo_projection} provides an efficient oracle for minimizing of $f$ with first-order algorithms.

\section{A Python Toolbox for Superquantile Optimization}\label{sec:spqr}
We provide a Python software package called \texttt{SPQR} to the =community for research in superquantile-based optimization and learning. 
% \footnote{Following anonymization guidelines, name of the toolbox and link to the code and documentation will be revealed under conditions of acceptance or at the request of the reviewers.}
The software package includes optimization and modeling tools to solve problems of the form\;\eqref{eq:general_problem} with just a few lines of code. 
%The toolbox is named \texttt{SPQR} for SuPer Quantile Risk optimization.  
The implementation builds off basic structures of~\texttt{scikit-learn} \cite{scikit-learn}.% the popular python machine learning library. 
The code is publicly available at~\url{https://github.com/yassine-laguel/spqr}.

We describe here the optimization methods used in %\texttt{SPQR}
in the toolbox and how to call the basic functions. 
% and its basic usage (input format and execution) to emphasize its simplicity. 
We refer to the online documentation for more details, custom options, and parameter settings. %\textbf{Link to the code will be sent after acceptation.}
 
%=================================================================
\subsection{First-order optimization algorithms}
%=================================================================

Although stochastic gradient algorithms are popular methods to solve empirical risk minimization problems at scale~\eqref{eq:ERM}, replacing the expectation by the superquantile in \eqref{eq:ESM} completely changes the situation making these algorithms not directly applicable. %complicates
%the sampling techniques and 
Indeed computing the function values and gradients
requires sorting loss values on the whole data set, which is not directly amenable to classical stochastic gradient algorithms.
This rehabilitates batch optimization algorithms in our context. We cover a variety of methods
%The toolbox 
%uses  first-order algorithms:
%for solving\;\eqref{eq:general_problem}: 
\begin{itemize}
    \setlength\itemsep{1mm}
    \item when $L$ is convex: subgradient method and dual averaging.
    % Start Description
    We implement in particular the ``weighted'' version of dual averaging with a Euclidean prox-function~\cite{nesterov2009primal}. For an iterate $x_k$ and a gradient $g_k$ of $f$ at $x_k$, the update writes:
    \[
        x_{k+1} = \frac{-s_{k+1}}{\alpha_k}\;\; \text{with}\;\;s_{k+1}=\sum_{i=0}^{k}\frac{g_k}{\|g_k\|} 
    \]
    where $(\alpha_k)_{k\geq 0}$ denotes the tuned step-size of the method. The tuning is carried through a line-search strategy performed at the first iteration.
    % End Description 
    To use these algorithms, we provide a subgradient oracle (from Proposition\;\ref{thm:sub}) with the same complexity as computing a quantile (ie.\;$\mathcal{O}(n)$ with $n$ the number of data points).
    \item when $L$ is smooth, we can use the smoothed objective: gradient method, accelerated gradient method and quasi-Newton (BFGS). 
    % Start Description
    In particular the accelerated gradient method relies on the following scheme\;\cite{10029946121}:
    \[
    \begin{split}
        \alpha_0 &= 0,\; \alpha_s = \frac{1 + \sqrt{1 + 4\alpha_{s-1})}}{2} \text{ and } \gamma_s = \frac{1 - \alpha_s}{\alpha_{s+1}}\\
        x_{s+1} &= y_s - \frac{1}{\beta} \nabla f(y_s), \;
        y_{s+1} = (1-\gamma_s)x_{s+1} + \gamma_s x_s \\
    \end{split}
    \]
    with $x_0=y_0=0$.
    % End Description
    To use these algorithms, we provide a gradient oracle using Algorithm\;\ref{algo_projection}, again with a $\mathcal{O}(n)$ complexity (Proposition\;\ref{thm:On}).
\end{itemize}
% (i) subgradient method and dual averaging for the general nonsmooth case %\eqref{sec:sub-grad-expr}
% and (ii) . 
% The smoothing implemented is the euclidean smoothing : $d(q) = \frac{1}{2} \|q - \frac{e}{n}\|_2^2$.
%. The default algorithm is the subgradient which handles the general case. 

%=================================================================
\subsection{Basic usage: input format and execution}
%=================================================================
%\texttt{SPQR} has a simple 

The user provides a dataset %modeled as a couple
$(X,Y)\in \mathbb{R}^p\times\mathbb{R}^m$ and an 
%first order 
oracle for the function $L$ and its gradient. The dataset is stored into two python lists (or numpy arrays) \texttt{X} and \texttt{Y}; for instance, %they can be arrays of 
for realizations of random variables:
\begin{lstlisting}
    import numpy as np
    X = np.random.rand(100, 2)
    alpha = np.array([1., 2.])
    Y = np.dot(X, alpha) + np.random.rand(100)
\end{lstlisting}
The two python functions \texttt{L} and \texttt{L\_prime} are assumed to be functions of the triplet \texttt{(w,x,y)} where \texttt{w} is the optimization variable and \texttt{(x,y)} a data point. For instance, one can perform risk-sensitive linear regression with:
\begin{lstlisting}
    # Define the loss and its derivative
    def L(w,x,y):
        return 0.5 * np.linalg.norm(y - np.dot(x,w))**2
    def L_prime(w,x,y):
        return -1.0 * (y - np.dot(x,w)) * x 
\end{lstlisting}
Before solving the problem \eqref{eq:general_problem}, we have to instantiate the \texttt{RiskOptimizer} object of \texttt{SPQR} with the two oracles, following standard usage of \texttt{scikit-learn}. The basic instantiation is as follows.
\begin{lstlisting}
    from SPQR import RiskOptimizer
    # Instantiate a risk optimizer object
    optimizer = RiskOptimizer(L, L_prime)
\end{lstlisting}
\texttt{RiskOptimizer} inherits from \texttt{scikit-learn}'s estimators: we use the \texttt{fit} method to run the optimization algorithm on the %provided 
data, %The optimization is launched with the next commands 
providing a solution of\;\eqref{eq:general_problem}.
\begin{lstlisting}
    #  Running the algorithm
    optimizer.fit(X,Y)
    lst_iterates = optimizer.list_iterates
    sol = optimizer.solution 
\end{lstlisting}

\section{Numerical Illustrations}\label{sec:numexp}
We compare the proposed approach~\eqref{eq:ESM} with the common approach using empirical risk minimization on synthetic and real data. 
% We compare the solution of our superquantile model with the one of the 
We solve the ordinary least squares problem
\begin{equation*}%\label{eq:OLS}
    \min_{w\in \mathbb{R}^d} \mathbb{E}_{(x_i, y_i) \sim \mathcal{D}}\big((y_i - w^\top x_i)^2\big)
\end{equation*}
using the corresponding function of \texttt{scikit-learn}
(by calling \texttt{LinearRegression.fit(X,Y)} method). We solve its risk-sensitive counterpart
\begin{equation*}%\label{eq:SQLS}
    \min_{w\in \mathbb{R}^d} [\bar Q_p]_{(x_i, y_i) \sim \mathcal{D}}\big((y_i - w^\top x_i)^2\big)
\end{equation*}
using our toolbox with risk-sensitive linear regression, Euclidean smoothing (with $\mu = 1000$), and
L-BGFS as optimizer (see Sec.~\ref{sec:spqr}). 

% The first-order algorithm used for the solving is L-BGFS with a smoothing parameter $\mu = 1000$ for the smoothing of the superquantile. 

\subsection{Synthetic Dataset}\label{sec:synthetic}

We consider a regression task on a synthetic training dataset %$(x_i,y_i)_{1\leq i \leq n}$ 
of $n=10^4$ points in $\mathbb{R}^{40} \times \mathbb{R}$. The design matrix $X = (x_i)_{1\leq i\leq n}$ is generated with the \texttt{make\_low\_rank\_matrix} procedure of \texttt{scikit\_learn}\;\cite{scikit-learn} with a rank $30$. For a given model parameter $\bar w \in \mathbb{R}$, we generate the data according to
% $(y_i)_{1\leq i\leq n}$, we 
%and let $y_i$ be defined as:
\begin{equation*}
    y_i = x_i^\top \bar w + \varepsilon_i.
\end{equation*}
The noise $\varepsilon_i$ is defined here as a mixture
%sampled from a mixture of two distributions:
% \[
%     \varepsilon_i \sim \mathcal{B}\; \varepsilon_{\mathcal{N}} +  (1-\mathcal{B}) \; \varepsilon_{\mathcal{L}}
% \]
\[
    \varepsilon_i = \beta\varepsilon_{\mathcal{N}} +  (1-\beta)\varepsilon_{\mathcal{L}}
\]
where all random variables are independent, $\varepsilon_{\mathcal{N}}$ follows a standard normal distribution, 
$\varepsilon_{\mathcal{L}}$ follows a Laplace distribution with location $\mu=10$ and scale $s=1$, 
and $\beta$ follows a Bernoulli distribution with parameter $p=0.8$. Define the \emph{squared residuals} (or losses) 
\[
r_i^2 = (y_i - w^\top x_i)^2 \quad \text{for } i=1,\ldots, n
\]
and the $p$-quantiles of the empirical distribution of $(r_i^2)_{i=1,\dots,n}$ for $p=0.5$ and $p=0.9$.

% We report the distribution of losses $(y_i - w^\top x_i)^2$ for $2000$ testing data point $(x_i,y_i)$
% (independently generated with the same procedure). 

\begin{table}[ht]
\vspace{-1mm}
\begin{center}
\begin{adjustbox}{max width=0.9\linewidth}
\begin{tabular}{lccccccc}
\toprule
Model & Mean & \multicolumn{2}{c}{$p$-quantile of the loss} \\
 &  & $p=0.5$ & $p=0.9$\\
\midrule
$\mathbb{E}$ & $16.45$ & $5.55$ &  $60.2$ \\
% $\bar Q_p\;-\;p=0.2$ & $16.5$ & $1.04$ & $6.57$ &  $57.0$ \\
$\bar Q_p\;-\;p=0.5$ & $18.75$  & $13.9$ &  $41.2$ \\
$\bar Q_p\;-\;p=0.7$ & $22.3$  & $20.7$ &  $36.6$ \\
$\bar Q_p\;-\;p=0.9$ & $23.7$ & $22.5$ &  $37.7$ \\
\bottomrule
\end{tabular}
\end{adjustbox}
\vspace{-1mm}
\caption{Quantiles of the empirical distribution of residuals on the test.~\label{table:synthetic_exp}}
\end{center}
\end{table}

\begin{figure}[ht]
\vspace{-8mm}
\begin{minipage}[b]{1.0\linewidth}
  \centering
  \centerline{\includegraphics[width=8.5cm]{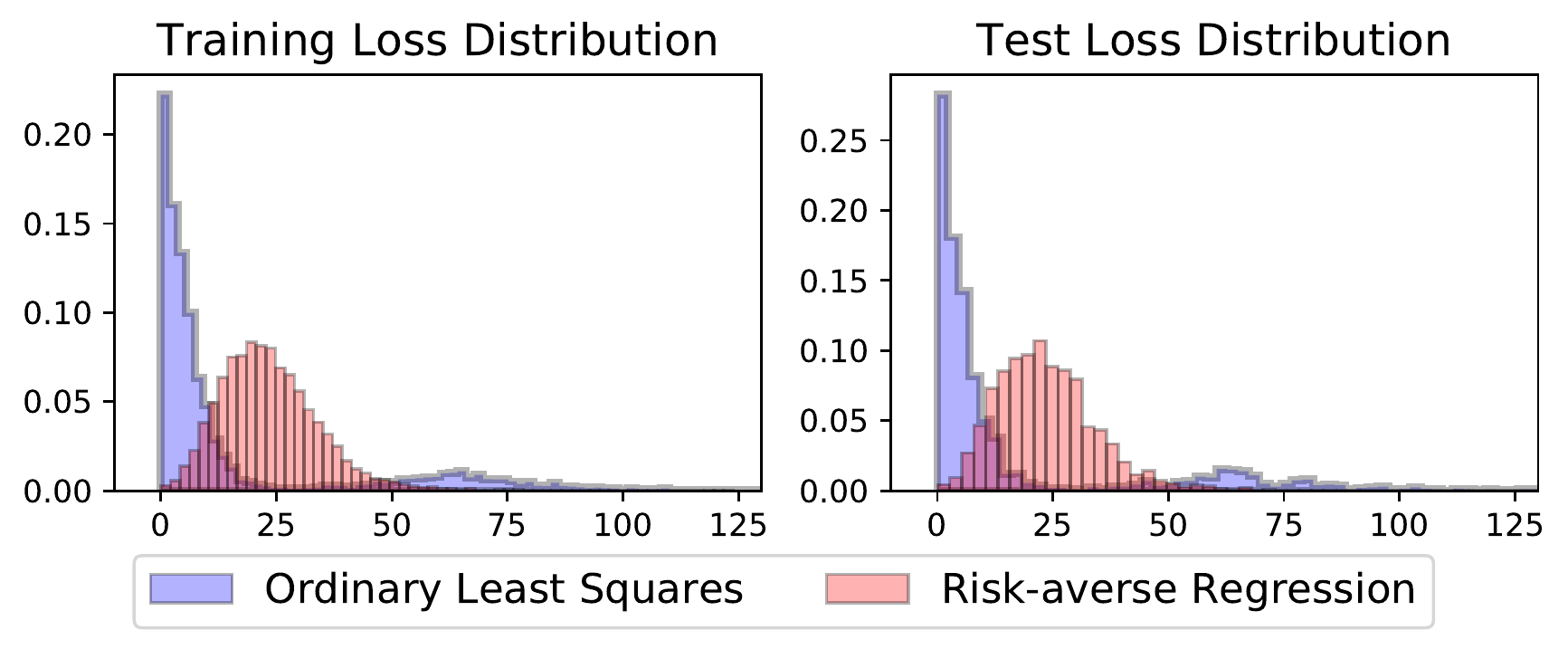}}
%  \vspace{2.0cm}
%   \centerline{ Architecture of the Neural Network for the dataset MNIST}\medskip
\end{minipage}
\vspace{-7mm}
\caption{Quantiles of the empirical distribution of residuals on the test. The risk-sensitive model was trained with $p=0.9$.}
\label{fig:synthetic_dataset}
\end{figure}

We report the $p$-quantiles and the distribution of losses obtained on the training dataset
%at the optimum found for each model. 
%We additionally report the performances of each model 
and on a test dataset of $2000$ data points independently generated with the same procedure; see Table\;\ref{table:synthetic_exp} and Figure\;\ref{fig:synthetic_dataset}. As $p$ grows, the superquantile-based or risk-sensitive model %exhibits thinner upper tails on the error, 
shifts the upper tail on errors to the left, 
which shows an improved performance on extreme cases. This comes with the price of % improvement is mitigated by  the
lower performances %of the superquantile approach 
on inputs well managed by the standard %ordinary least squares 
approach (see metrics for $p=0.5$ in Table\;\ref{table:synthetic_exp}).
%and red bell on Figure\;\ref{fig:synthetic_dataset}).

\subsection{Real Dataset}

We consider the superconductivity dataset \cite{hamidieh2018data} which contains the information of $21,263$ superconductors.  The learning task is to predict the critical temperature of a superconductor from the $10$ most important features as selected by \cite{hamidieh2018data}.
% Among the $81$ proposed features, we select $10$ of them put forward in \cite{hamidieh2018data}. 
We split the dataset into a training set and a testing set with a ratio $80\%/20\%$. 

%We solve the Ordinary Least Squares problem and its risk-sensitive counterpart on the training set.
We report in Figure \ref{fig:real_dataset} the comparison between the quantiles of the testing and training loss distribution respectively. In terms of the quantile at $90\%$, the proposed approach display better statistical behavior on the testing loss than the common approach based on empirical risk minimization. This is in line with the aim of the formulation considered, which seeks to gain a better control on the tails of the loss distribution. 
\begin{table}[ht]
\label{table:real_exp}
% \vskip 0.15in
\begin{center}
% \vspace{-2mm}
\begin{adjustbox}{max width=0.9\linewidth}
\begin{tabular}{lccccccc}
\toprule
Model & Mean & \multicolumn{3}{c}{$p$-quantile of the loss} \\
 & & $p=0.9$ & $p=0.95$ & $p=0.99$\\
\midrule
$\mathbb{E}$ & $16.5$ & $35.8$ & $42.7$ &  $55.7$ \\
$\bar Q_p\;-\;p=0.8$ & $17.4$ & $34.7$ & $41.0$ &  $53.8$ \\
$\bar Q_p\;-\;p=0.9$ & $18.1$ & $35.6$ & $41.0$ &  $53.6$ \\
$\bar Q_p\;-\;p=0.95$ & $18.9$ & $36.5$ & $41.4$ &  $53.6$ \\
% $\bar Q_p\;-\;p=0.99$ & $24.1$ & $41.2$ & $46.3$ &  $54.7$ \\
\bottomrule
\end{tabular}
\end{adjustbox}
\end{center}
\vspace{-4mm}
\caption{\small{Metrics of the distribution of the loss values $r_i$ on the test superconductivity dataset}}
\end{table}

\begin{figure}[htb]
\vspace{-2mm}
\begin{minipage}[b]{1.0\linewidth}
  \centering
  \centerline{\includegraphics[width=8.5cm]{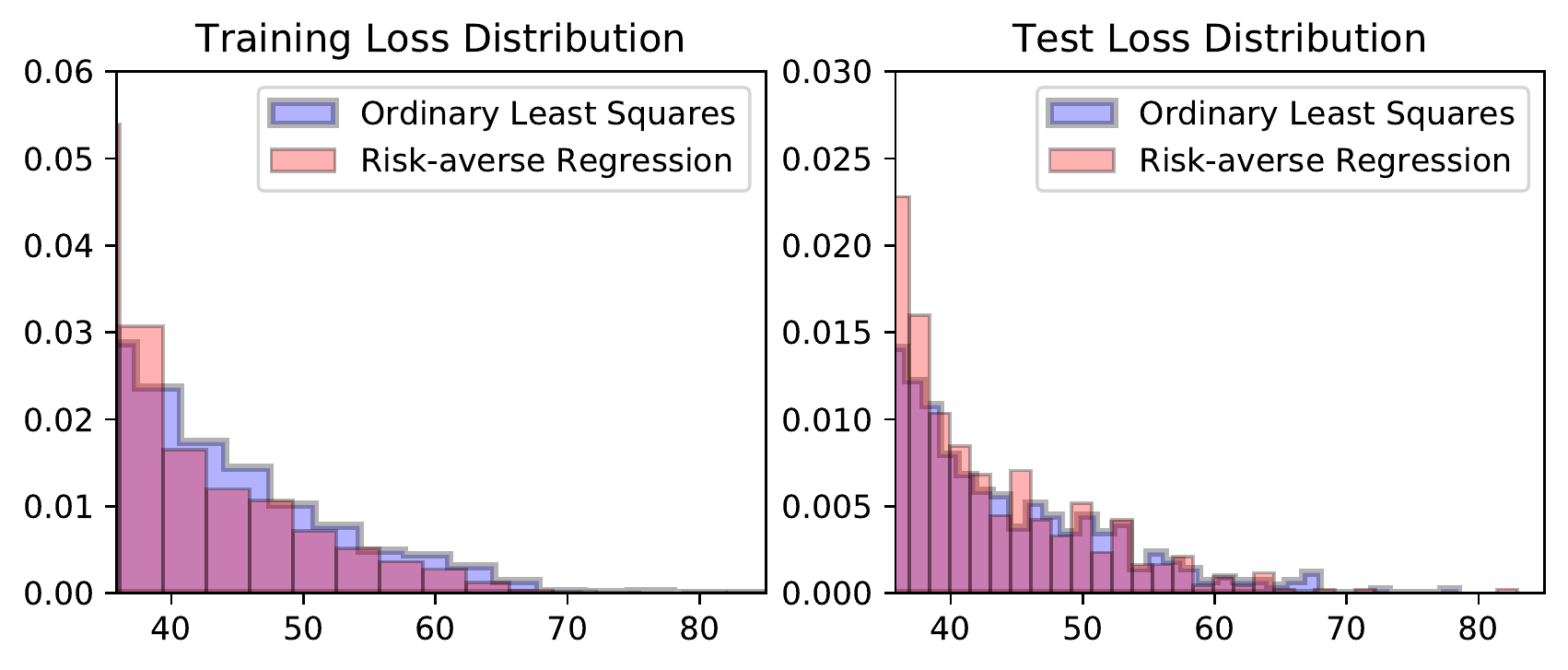}}
%  \vspace{2.0cm}
%   \centerline{ Architecture of the Neural Network for the dataset MNIST}\medskip
% \vspace{-1mm}
\end{minipage}
\caption{Distribution of the loss values $r_i$ on the train and test superconductivity dataset. The risk-sensitive model is trained with $p=0.9$.}
\label{fig:real_dataset}
\end{figure}

\section{Conclusion}
Risk-sensitive optimization plays a major role in the design of safer models for decision-making and has recently gained interest in machine learning. We provide a %publicly available 
toolbox to tackle superquantile-based learning problems using first-order optimization algorithms.
%We explain the theory on which it is grounded, the key numerical ingredients, and how to use it. 
Numerical illustrations on regression tasks show an improved statistical behavior in terms of higher quantiles of the testing loss.
%illustrating the minimization of worst-case data. 

%==============================
% Example Figure
%==============================
% \begin{figure}[htb]

% \begin{minipage}[b]{1.0\linewidth}
%   \centering
%   \centerline{\includegraphics[width=8.5cm]{image1}}
% %  \vspace{2.0cm}
%   \centerline{(a) Result 1}\medskip
% \end{minipage}
% %
% \begin{minipage}[b]{.48\linewidth}
%   \centering
%   \centerline{\includegraphics[width=4.0cm]{image3}}
% %  \vspace{1.5cm}
%   \centerline{(b) Results 3}\medskip
% \end{minipage}
% \hfill
% \begin{minipage}[b]{0.48\linewidth}
%   \centering
%   \centerline{\includegraphics[width=4.0cm]{image4}}
% %  \vspace{1.5cm}
%   \centerline{(c) Result 4}\medskip
% \end{minipage}
% %
% \caption{Example of placing a figure with experimental results.}
% \label{fig:res}
% %
% \end{figure}

% \vfill
% \pagebreak

% \bibliographystyle{IEEEbib}
% \bibliography{strings,refs}
\bibliographystyle{IEEEbib}
\bibliography{optim,refs}

\begin{thebibliography}{10}

\bibitem{recht2019imagenet}
Benjamin Recht, Rebecca Roelofs, Ludwig Schmidt, and Vaishaal Shankar,
\newblock ``Do imagenet classifiers generalize to imagenet?,''
\newblock {\em arXiv:1902.10811}, 2019.

\bibitem{metz2018microsoft}
Rachel Metz,
\newblock ``Microsoft's neo-{Nazi} sexbot was a great lesson for makers of {AI}
  assistants,''
\newblock {\em Artificial Intelligence}, March 2018.

\bibitem{knight2018selfdriving}
Will Knight,
\newblock ``A self-driving {Uber} has killed a pedestrian in {Arizona},''
\newblock {\em Ethical Tech}, March 2018.

\bibitem{rockafellar2008risk}
R.T. Rockafellar, S.~Uryasev, and M.~Zabarankin,
\newblock ``Risk tuning with generalized linear regression,''
\newblock {\em Mathematics of Operations Research}, 2008.

\bibitem{lee2018minimax}
J.~Lee and M.~Raginsky,
\newblock ``Minimax statistical learning with {W}asserstein distances,''
\newblock in {\em Advances in Neural Information Processing Systems}, 2018.

\bibitem{duchi2019variance}
J.~C. Duchi and H.~Namkoong,
\newblock ``Variance-based {R}egularization with {C}onvex {O}bjectives.,''
\newblock {\em Journal of Machine Learning Research}, 2019.

\bibitem{kuhn2019wasserstein}
D.~Kuhn, P.M. Esfahani, V.~Anh Nguyen, and S.~Shafieezadeh-Abadeh,
\newblock ``Wasserstein distributionally robust optimization: Theory and
  applications in machine learning,''
\newblock in {\em Operations Research \& Management Science in the Age of
  Analytics}. INFORMS, 2019.

\bibitem{ben2009robust}
A.~Ben-Tal, L.~El~Ghaoui, and A.~Nemirovski,
\newblock {\em Robust optimization},
\newblock Princeton University Press, 2009.

\bibitem{owen2001empirical}
A.B. Owen,
\newblock {\em Empirical Likelihood},
\newblock Chapman \& Hall/CRC Monographs on Statistics \& Applied Probability.
  CRC Press, 2001.

\bibitem{rockafellar2000optimization}
T.~Rockafellar and S.~Uryasev,
\newblock ``Optimization of {C}onditional {V}alue-at-{R}isk,''
\newblock {\em Journal of Risk}, 2000.

\bibitem{ben2007old}
A.~Ben-Tal and M.~Teboulle,
\newblock ``An old-new concept of convex risk measures: The optimized certainty
  equivalent,''
\newblock {\em Mathematical Finance}, 2007.

\bibitem{shapiro2014lectures}
A.~Shapiro, D.~Dentcheva, and A.~Ruszczy{\'n}ski,
\newblock {\em Lectures on stochastic programming: modeling and theory},
\newblock SIAM, 2014.

\bibitem{rockafellar2013superquantiles}
R.~T. Rockafellar and J.~O Royset,
\newblock ``Superquantiles and their applications to risk, random variables,
  and regression,''
\newblock in {\em Theory Driven by Influential Applications}. INFORMS, 2013.

\bibitem{follmer2002convex}
H.~F{\"o}llmer and A.~Schied,
\newblock ``Convex measures of risk and trading constraints,''
\newblock {\em Finance and stochastics}, 2002.

\bibitem{rockafellar2014superquantile}
R.T. Rockafellar, J.O. Royset, and S.I. Miranda,
\newblock ``Superquantile regression with applications to buffered reliability,
  uncertainty quantification, and conditional value-at-risk,''
\newblock {\em European Journal of Operational Research}, 2014.

\bibitem{nesterov2005smooth}
Y.~Nesterov,
\newblock ``Smooth minimization of non-smooth functions,''
\newblock {\em Mathematical programming}, 2005.

\bibitem{ruszczynski2006optimization}
A.~Ruszczy{\'n}ski and A.~Shapiro,
\newblock ``Optimization of convex risk functions,''
\newblock {\em Mathematics of operations research}, 2006.

\bibitem{hiriart2013convex}
J.-B. Hiriart-Urruty and C.~Lemar{\'e}chal,
\newblock {\em Convex analysis and minimization algorithms I: Fundamentals},
\newblock Springer science \& business media, 2013.

\bibitem{doi:10.1137/100818327}
A.~Beck and M.~Teboulle,
\newblock ``Smoothing and first order methods: A unified framework,''
\newblock {\em SIAM Journal on Optimization}, 2012.

\bibitem{condat2016fast}
L.~Condat,
\newblock ``Fast projection onto the simplex and the l1 ball,''
\newblock {\em Mathematical Programming}, 2016.

\bibitem{scikit-learn}
F.~Pedregosa et~al.,
\newblock ``Scikit-learn: Machine learning in {P}ython,''
\newblock {\em Journal of Machine Learning Research}, 2011.

\bibitem{nesterov2009primal}
Y.~Nesterov,
\newblock ``Primal-dual subgradient methods for convex problems,''
\newblock {\em Mathematical programming}, 2009.

\bibitem{10029946121}
Y.~Nesterov,
\newblock ``A method for solving the convex programming problem with
  convergence rate $\mathcal{O}(1/k^2)$,''
\newblock {\em Dokl. Akad. Nauk SSSR}, 1983.

\bibitem{hamidieh2018data}
K.~Hamidieh,
\newblock ``A data-driven statistical model for predicting the critical
  temperature of a superconductor,''
\newblock {\em Computational Materials Science}, 2018.

\end{thebibliography}

\end{document}